\documentclass[12pt,final]{article}
\usepackage{amsmath,amsthm,amsfonts,amssymb,graphicx,enumerate,psfrag}
\usepackage{tikz,pgfplots}
\usepackage{mathrsfs}
\usepackage{fullpage}

\usepackage[colorlinks,citecolor=blue,urlcolor=blue]{hyperref}
\definecolor{myblue}{RGB}{51,51,178}
\definecolor{myred}{RGB}{189,26,26}
\definecolor{mygreen}{RGB}{0,128,0}
\usepackage[utf8]{inputenc} 
\newtheorem{lemma}{Lemma}

\newtheorem{definition}{Definition}
\newtheorem{theorem}{Theorem}

\newtheorem{corollary}{Corollary}

\newtheorem{conjecture}{Conjecture}

\newcommand{\dN}{\mathbb {N}}

\newcommand{\cQ}{\mathcal {Q}}

\newcommand{\dR}{\mathbb {R}}
\newcommand{\CD}{\mathrm{CD}}

\newcommand{\dist}{{\mathrm{dist}}}

\title{Edge-regular graphs with non-negative curvature have  polynomial growth}
\author{Guy Blachar, Hervé Pajot and Justin Salez}
\begin{document}
\maketitle
\begin{abstract}
A long-standing conjecture in the emerging discrete Bakry-\'Emery theory  asserts that bounded-degree graphs satisfying $\CD(0,\infty)$ have polynomial growth. In the present paper, we prove this conjecture for all edge-regular graphs, and even obtain a volume doubling estimate with a constant that depends only on the degree. This is made possible thanks to the discovery of a surprising self-improvement phenomenon, which seems of independent interest: any edge-regular graph satisfying $\CD(\kappa,\infty)$ for some $\kappa\in\dR$  must in fact satisfy $\CD(\kappa,n)$ for some explicit, universal and optimal dimension parameter $n$. 
\end{abstract}

\section{Introduction}
Let  $G=(V,E)$ be a locally finite, connected, simple, undirected  graph, and let $\Delta$ denote its Laplacian operator, which  acts on any function $f\in\dR^V$ as follows:
\begin{eqnarray}
\label{def:Delta}
\forall x\in V,\qquad \Delta f(x) & := & \sum_{y\in N_x}\left(f(y)-f(x)\right),
\end{eqnarray}
where $N_x:=\{y\in V\colon\{x,y\}\in E\}$ is the set of neighbors of $x$. Following \cite{MR889476,MR1665591,MR2644381}, we define the carré du champ operator $\Gamma_1$ and its iterated version $\Gamma_2$ inductively via
\begin{eqnarray}
\Gamma_{k}(f,g) & := & \frac{1}{2}\left(\Delta \Gamma_{k-1}(f,g)-\Gamma_{k-1}(f,\Delta g)-\Gamma_{k-1}(g,\Delta f)\right),
\end{eqnarray} 
for all $f,g\in\dR^V$, with $\Gamma_0(f,g):=fg$. Given two parameters $\kappa\in\dR$ and $n\in(0,\infty]$, we then say that  $G$ satisfies the \emph{Curvature-Dimension} condition $\CD(\kappa,n)$ if 
\begin{eqnarray}
\label{def:CD}
\forall f\in\dR^V,\quad 
\Gamma_2(f,f) & \ge & \kappa\, \Gamma_1(f,f) + \frac{(\Delta f)^2}{n}.
\end{eqnarray}As observed by Bakry and \'Emery in their seminal work \cite{MR889476}, this pointwise inequality holds when  $\Delta$ is the Laplace-Beltrami operator of a $n-$dimensional Riemannian manifold  with Ricci curvature at least $\kappa$, and it turns out to have a number of remarkable analytical, geometric and probabilistic consequences for the associated diffusion. We refer the unfamiliar reader to the excellent textbook \cite{MR3155209} for an introduction to this beautiful theory, and to  \cite{salez2025modernaspectsmarkovchains,Salez2026Cutoff,pedrotti2025newcutoffcriterionnonnegatively} for recent connections with the cutoff phenomenon. 

In recent years, considerable efforts have been made to develop a discrete analogue of the Bakry-\'Emery theory, where manifolds are replaced by graphs  \cite{MR1665591,MR2644381,MR3173151,MR3492631, MR3592766,MR3776357,MR3906157,MR4045968,MR4946365,MR4545901,MR4828188,salez2025intrinsicregularitydiscretelogsobolev}. Unfortunately, several important pieces   are still missing, due to the notorious failure of the chain rule for differentiating composed functions on discrete spaces. In particular, the following fundamental relation between curvature and volume is still lacking, and constitutes one of the most frustrating gaps between the Riemannian Bakry-\'Emery theory and its emerging discrete  counterpart. Throughout the paper, we write $\dist(\cdot,\cdot)$ for the usual graph distance, and  $B(x,r)$ for the corresponding open ball of radius $r$ around $x$: 
\begin{eqnarray}
B(x,r) & := & \left\{y\in V\colon \dist(x,y)<r\right\}. 
\end{eqnarray}
\begin{conjecture}[Non-negative curvature implies polynomial growth]\label{conj:main} If $G$ has bounded degrees and satisfies $\CD(0,\infty)$, then $G$ has polynomial growth. More precisely,
\begin{eqnarray}
\label{def:poly}
\forall x\in V,\quad \forall r\in \dN,\quad \# B(x,r) & \le & r^D,
\end{eqnarray}
where the constant $D<\infty$  depends only on the maximum degree. 
\end{conjecture}
The degree restriction is of course necessary here, since choosing $r=2$ in (\ref{def:poly}) shows that $D$ is at least logarithmic in the maximum degree. 
Conjecture \ref{conj:main} has received a lot of attention \cite{hua2017ricci,münch2019,russ2025}, and several results have been establish in its direction \cite{MR3316971,MR4536468,MR4567745,hutchcroft2025}. In particular, a major breakthrough was recently obtained by Hutchcroft  and M\"unch \cite{hutchcroft2025}, who proved that bounded-degree graphs with non-negative  curvature in the (different) sense of Ollivier \cite{MR2484937} have sub-exponential growth. Another remarkable result in the direction of Conjecture \ref{conj:main} is the fact that polynomial growth is indeed guaranteed when the assumption  $\CD(0,\infty)$ is strengthened to its finite-dimensional refinement $\CD(0,n)$, for any  $n<\infty$.  More precisely, the following \emph{Volume Doubling} estimate was recently obtained by Münch \cite{münch2019} for finite graphs and by Pajot and Russ \cite{russ2025} in the infinite case (see also \cite{MR3316971} for a closely related,  wearker result).
\begin{theorem}[$\CD(0,n)$ implies Volume Doubling, c.f. \cite{münch2019,russ2025}]\label{th:dim} If $G$ has bounded degrees and satisfies $\CD(0,n)$ for some $n<\infty$, then 
\begin{eqnarray}
\label{def:VD}
\forall x\in V,\quad \forall r\in\dN,\quad \# B(x,2r) & \le & C\, \# B(x,r),
\end{eqnarray}
where the constant $C<\infty$ depends only on $n$ and the maximum degree. 
\end{theorem}
Note that this volume doubling estimate can readily be iterated to yield a polynomial growth estimate of the form (\ref{def:poly}), with $D=1+\log_2 C$. Unfortunately, the constant $C$ obtained in \cite{münch2019,russ2025} tends to infinity with $n$, and new ideas seem to be needed to reach Conjecture \ref{conj:main}. 

In the present paper, we establish Conjecture \ref{conj:main} for all edge-regular graphs. Let us here recall that $G$ is \emph{edge-regular} if there exist two integers $d,\tau\ge 0$ such that any vertex has exactly $d$ neighbors, and any two adjacent vertices have exactly $\tau$ common neighbors:
\begin{eqnarray}
\label{def:edgereg}
\forall x\in V,\quad \# N_x\ =\ d, & \textrm{ and } & \forall \{x,y\}\in E,\quad  \#(N_x\cap N_y)\ =\ \tau.
\end{eqnarray}
For example, any  triangle-free regular graph is edge-regular, as well as any Cayley graph generated by a conjugacy class. Our main contribution is the surprising discovery that an edge-regular graph satisfying $\CD(0,\infty)$ must actually  satisfy the stronger property $\CD(0,d)$, where $d$ is the degree. Our result is in fact substantially stronger: for any value of the curvature parameter $\kappa\in\dR$, we determine the optimal value of the dimension parameter $n$ such that $\CD(\kappa,\infty)$ is equivalent to $\CD(\kappa,n)$, simultaneously for all graphs satisfying (\ref{def:edgereg}).
\begin{theorem}[Self-improvement of the Bakry-\'Emery condition on edge-regular graphs]\label{th:main}Fix $\kappa\in\dR$ and two integers $d,\tau\ge 0$, and suppose that $G$ satisfies $\CD(\kappa,\infty)$ and (\ref{def:edgereg}). Then, 
\begin{enumerate}
\item  $2\kappa\le 4+\tau$; 
\item $\CD(\kappa,n)$ holds for $n:=4d/(4+\tau-2\kappa)$.
\item $\CD(\kappa,n')$ fails for any $n'<n$. 
\end{enumerate}
\end{theorem}
In particular, any edge-regular graph satisfying $\CD(0,\infty)$ automatically satisfies $\CD(0,d)$ where $d$ is the degree, as promised. Combining this with Theorem \ref{th:dim}, we deduce that Conjecture \ref{conj:main} holds, in a strong sense, for all edge-regular graphs. 
\begin{corollary}[Volume doubling for edge-regular graphs with non-negative curvature]If $G$ is edge-regular and satisfies $\CD(0,\infty)$, then the Volume Doubling property (\ref{def:VD}) holds with a constant $C$ that depends only on the degree. In particular,  $G$ satisfies Conjecture \ref{conj:main}. 
\end{corollary}
Whether the edge-regularity assumption in  Theorem \ref{th:main} can be relaxed is a natural question, with direct implications for Conjecture \ref{conj:main}. We note that at least \emph{some} restrictions are necessary. Indeed, in Section \ref{sec:counterexample}, we will construct a $7-$vertex graph satisfying $\CD(0,\infty)$ but not $\CD(0,n)$ for any $n<\infty$. 
We were not able to produce a regular graph with those properties, and we therefore propose the following natural strengthening of Theorem \ref{th:main}. 
\begin{conjecture}
Any $d-$regular graph satisfying $\CD(0,\infty)$ satisfies $\CD(0,d)$.
\end{conjecture}

\section{Proof of the main result}
 We will naturally say that the $\CD(\kappa,n)$ condition holds at a vertex $o\in V$ if 
\begin{eqnarray}
\label{def:locCD}
\Gamma_2(f,f)(o) & \ge &  \kappa\,\Gamma_1(f,f)(o) + \frac{(\Delta f(o))^2}{n},
\end{eqnarray}
for all  $f\in \dR^V$. Our proof of Theorem \ref{th:main} relies on the construction of \emph{height functions}. 
\begin{definition}[Height functions]Given  two parameters $\kappa,\lambda\in\dR$ and a vertex $o\in V$, we say that $g\in\dR^V$ is a $(\kappa,\lambda)-$height function at $o$ if $\Delta g(o)  =  1$ and for all $f\in \dR^V$,
\begin{eqnarray*}
\Gamma_2(f,g)(o) & = & \kappa\,\Gamma_1(f,g)(o) + \lambda\, \Delta f(o).
\end{eqnarray*}
\end{definition}
The interest of this mysterious definition is contained in the following result, which shows that height functions systematically provide an optimal self-improvement in the local Bakry-\'Emery criterion. As the reader will see, the proof does not use the specific form of the discrete Laplacian (\ref{def:Delta}): it applies to any Markov generator $\Delta$ on any state space. The converse part will not be used here, but it guarantees that there is no intrinsic limitation in using height functions as a general tool to establish curvature-dimension estimates. 
\begin{lemma}[Height functions yield optimal Bakry-\'Emery self-improvement]\label{lm:main} If $\CD(\kappa,\infty)$ holds at $o$ and  if there exists a $(\kappa,\lambda)-$height function at $o$, then
\begin{enumerate}
\item  $\lambda\ge 0$; 
\item $\CD(\kappa,n)$ holds at $o$, with $n:=1/\lambda$.
\item $\CD(\kappa,n')$ fails at $o$, for any $n'<n$. 
\end{enumerate}
Conversely, if  $\CD(\kappa,n)$ holds at $o$ for some $n<\infty$ and if $V$ has at least two points, then there exists a $(\kappa,\lambda)-$height function at $o$ for some $\lambda \ge 1/n$.
\end{lemma}
\begin{proof}
The fact that $\CD(\kappa,\infty)$ holds at $o$ ensures that the bilinear symmetric form
\begin{eqnarray}
\label{def:cQ}
\cQ(f,g) & := & \Gamma_2(f,g)(o)-\kappa\,\Gamma_1(f,g)(o),
\end{eqnarray}
is positive semi-definite. Thus, it satisfies $\cQ(g,g)\ge 0$ and the Cauchy-Schwartz inequality:
\begin{eqnarray}
\label{CS}
\cQ^2(f,g) & \le & \cQ(f,f)\cQ(g,g),
\end{eqnarray}
for all $f,g\in \dR^V$. On the other hand, if $g$ is a $(\kappa,\lambda)-$height function at $o$, then
\begin{eqnarray*}
\cQ(g,g) \ = \ \lambda, & \textrm{ and } & \cQ(f,g) \ = \ \lambda\,\Delta f(o),
\end{eqnarray*}
for any $f\in \dR^V$. Thus, we must have $\lambda\ge 0$, and (\ref{CS}) becomes
\begin{eqnarray}
\label{CD:check}
\lambda\, \left(\Delta f(o)\right)^2 & \le & \cQ(f,f).
\end{eqnarray}
But this precisely means that $\CD(\kappa,n)$ holds at $o$, with $n:=1/\lambda$. Moreover, this value is optimal, because (\ref{CD:check}) is an equality when $f=g$. Conversely, if $\CD(\kappa,n)$ holds at $o$ for some $n<\infty$, then the linear form $f\mapsto \Delta f(o)$ is continuous w.r.t. the semi-norm induced by the positive semi-definite  bilinear symmetric form (\ref{def:cQ}), so the Riesz representation theorem ensures the existence of $g\in\dR^V$ such that 
\begin{eqnarray}
\forall f\in\dR^V,\quad \Delta f(o) & = & \cQ(f,g).
\end{eqnarray}
Note that we must then have $\cQ(g,g)>0$, as otherwise the map $f\mapsto \Delta f(o)$  would be identically zero, contradicting the fact that $G$ is connected  with at least two points. Thus, we can safely divide $g$ by $\cQ(g,g)$ to obtain a $(\kappa,\lambda)-$height function, where $\lambda=1/\cQ(g,g)$. But then, our assumption that  $\CD(\kappa,n)$ holds at $o$ forces $n\ge 1/\lambda$, by Item 2 in the lemma.
\end{proof}

In order to apply the above lemma, we now need to construct height functions. This is precisely the aim of the next lemma, in which the edge-regularity assumption is crucially used. 
When combined together, Lemmas \ref{lm:main} and \ref{lm:height} readily imply Theorem \ref{th:main}.   

\begin{lemma}[Existence of height functions]\label{lm:height}If  $G$ satisfies the edge-regularity property (\ref{def:edgereg}), then for every $\kappa\in\dR$ and every $o\in V$, there is a $(\kappa,\lambda)-$height function at $o$, with 
\begin{eqnarray}
\label{def:lambda}
\lambda & := & \frac{4+\tau-2\kappa}{4d}.
\end{eqnarray}
\end{lemma}
\begin{proof}
Given $o\in V$, we will show that the function  $g\colon x\mapsto \dist(o,x)$ satisfies
\begin{eqnarray}
(\Delta g)(o) & = & d,\\
2\Gamma_1(f,g)(o) & = & \Delta f(o),\\
4\Gamma_2(f,g)(o) & = &(4+\tau) \Delta f(o),
\end{eqnarray}
for all  $f\in \dR^V$. Clearly, the normalized function $g/d$ is then a $(\kappa,\lambda)-$height function, for any $\kappa\in\dR$ and $\lambda$ as in (\ref{def:lambda}). Now, the first two claims are immediate: we have
\begin{eqnarray}
\label{g}
\Delta g(o) & = & \sum_{x\in N_o}\left(g(x)-g(o)\right) \ = \ d,\\
\label{fg1}2\Gamma_1(f,g)(o) & = & \sum_{x\in N_o}\left(f(x)-f(o)\right)\left(g(x)-g(o)\right) \ = \ \Delta f(o),
\end{eqnarray}
for every  $f\in \dR^V$, because $g(x)-g(o)=1$ for all $x\in N_o$. To prove the third claim, we will now explicitate each of the three terms in the definition of $4\Gamma_2(f,g)(o)$:
\begin{eqnarray}
\label{fg2}
4\Gamma_2(f,g)(o) & = & 2\Delta\Gamma_1(f,g)(o)-2\Gamma_1(f,\Delta g)(o)-2\Gamma_1(\Delta f,g)(o).
\end{eqnarray}
The third term is readily obtained by replacing $f$ with $\Delta f$ in (\ref{fg1}):
\begin{eqnarray}
2\Gamma_1(\Delta f,g)(o) & = &  \Delta^2 f(o).
\end{eqnarray}
We now fix $x\in N_o$ and $y\in N_x$, and observe that $g(y)-g(x)$ is $-1$ if $y=o$, $0$ if $y\in N_x\cap N_o$, and $+1$ in the remaining cases. In view of our edge-regularity assumption, we deduce that
\begin{eqnarray*}
\Delta g(x) & = & \sum_{y\in N_x}\left(g(y)-g(x)\right) \ = \ d-2-\tau.
\end{eqnarray*}
Recalling that $\Delta g(o)=d$, we can now compute the second term in (\ref{fg2}):
\begin{eqnarray}
2\Gamma_1(f,\Delta g)(o) & = & \sum_{x\in N_o}\left(f(x)-f(o)\right)\left(\Delta g(x)-\Delta g(o)\right) \ = \ -(\tau+2)\Delta f(o).
\end{eqnarray}
Finally, in order to compute the first term in (\ref{fg2}), we fix again $x\in N_o$ and use the above discussion about the possible values of $g(y)-g(x)$ to write 
\begin{eqnarray*}
2\Gamma_1(f,g)(x) & = & \sum_{y\in N_x}\left(f(y)-f(x)\right)\left(g(y)-g(x)\right)\\
 & = & \Delta f(x)+2\left(f(x)-f(o)\right)+\sum_{y\in N_x\cap N_o}\left(f(x)-f(y)\right).
\end{eqnarray*}
Substracting (\ref{fg1}) and summing over all $x\in N_o$, we arrive at
 \begin{eqnarray*}
2\Delta\Gamma_1(f,g)(o) 
& = & \Delta^2f(o)+2\Delta f(o)+\sum_{x\in N_o}\sum_{y\in N_x\cap N_o}\left(f(x)-f(y)\right)\\
& = & \Delta^2f(o)+2\Delta f(o),
\end{eqnarray*}
where we have crucially observed that the function $(x,y)\mapsto f(x)-f(y)$  is anti-symmetric, and that the double sum runs over all pairs $(x,y)$ in $N_o^2\cap E$, which is a symmetric set. Having explicited each term in (\ref{fg2}), we can now add them up and simplify to obtain
\begin{eqnarray}
4\Gamma_2(f,g)(o) 
& = & (\tau+4)\Delta f(o),
\end{eqnarray}
as desired. 
\end{proof}
\section{A counter-example without edge-regularity}
\label{sec:counterexample}
In this final section, we construct an explicit finite graph satisfying  $\CD(0,\infty)$ but not $\CD(0,n)$ for any finite $n$, showing that our self-improvement result (Theorem \ref{th:main}) does require some structural assumptions. Note, however, that this does not refute Conjecture \ref{conj:main}.
\begin{lemma}[A  graph with no Bakry-\'Emery self-improvement]The $7-$vertex graph  depicted below satisfies  $\CD(0,\infty)$, but not $\CD(0,n)$ for any finite $n$: 
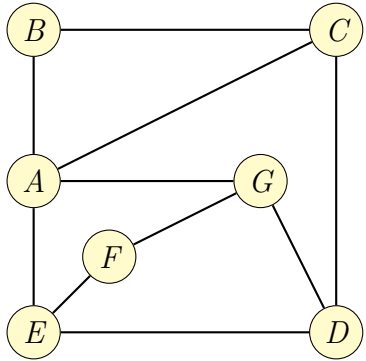
\begin{eqnarray*}
\begin{tikzpicture}[
    vertex/.style={circle, draw, fill=yellow!25, minimum size=20pt, inner sep=0pt},
    edge/.style={draw, thick}
]

% Vertices
\node[vertex] (0) at (6, 4) {C};
\node[vertex] (1) at (2, 0) {E};
\node[vertex] (2) at (5, 2) {G};
\node[vertex] (3) at (2, 4) {B};
\node[vertex] (4) at (3, 1) {F};
\node[vertex] (5) at (6, 0) {D};
\node[vertex] (6) at (2, 2) {A};

% Edges
\draw[edge] (0) -- (3);
\draw[edge] (0) -- (5);
\draw[edge] (0) -- (6);
\draw[edge] (1) -- (4);
\draw[edge] (1) -- (5);
\draw[edge] (1) -- (6);
\draw[edge] (2) -- (4);
\draw[edge] (2) -- (5);
\draw[edge] (2) -- (6);
\draw[edge] (3) -- (6);

\end{tikzpicture}
\end{eqnarray*}
\end{lemma}
\begin{proof}
Writing $\kappa(x)$ for the  largest number such that $\CD(\kappa(x),\infty)$ holds at  $x$, the Graph Curvature Calculator at \cite{MR4458133} provides the following values on the above graph:
\begin{eqnarray*}
\kappa(A) & = & 0 \\ 
\kappa(B) & = & 1.719\\
\kappa(C) & = & 1.325\\ 
\kappa(D) & = & 1\\
\kappa(E) & = & 1.083\\
\kappa(F) & = & 1.5\\
\kappa(G) & = & 1.083.
\end{eqnarray*}
The fact that all those values are non-negative precisely means that $\CD(0,\infty)$ holds. On the other hand, the particular function $f\in \dR^V$ defined by 
\begin{eqnarray*}
f(A) & = & 0\\
f(B) & = & -2\\
f(C) & = & -1\\
f(D) & = & 2\\
f(E) & = & 2\\
f(F) & = & 4\\
f(G) & = & 2,
\end{eqnarray*}
is easily checked to satisfy $\Delta f(A)=1$ and $\Gamma_2 f(A)=0$, which violates the local curvature-dimension condition $\CD(0,n)$ at the vertex $A$ for every $n<\infty$. 
\end{proof}

\section*{Acknowledgment}
This work is supported by the ERC consolidator grant CUTOFF (101123174). Views and
opinions expressed are however those of the authors only and do not necessarily reflect those
of the European Union or the European Research Council Executive Agency. 

\bibliographystyle{plain}
\bibliography{draft}

\end{document}